\documentclass{article}

\usepackage{graphicx}
\usepackage{amsfonts, amsmath, amssymb, amsthm, mathtools}
\usepackage{authblk}

\usepackage{xcolor}
\usepackage{xspace}

\usepackage{hyperref}
\usepackage{cleveref}

\usepackage[textwidth=1.5in]{todonotes}
\usepackage[margin=1in]{geometry}

\usepackage{thmtools}
\usepackage{thm-restate}

\usepackage{cancel}
\newcommand{\Cancel}[2][black]{{\color{#1}\cancel{\color{black}#2}}}

\theoremstyle{definition}

\newtheorem{environment}{Environment}[section]

\newtheorem{lemma}[environment]{Lemma}
\crefname{lemma}{lemma}{lemmata}
\crefformat{lemma}{#2Lemma~#1#3}
\Crefformat{lemma}{#2Lemma~#1#3}

\newtheorem*{lemma*}{Lemma}
\crefname{lemma*}{lemma}{lemmata}
\crefformat{lemma*}{#2Lemma~#1#3}
\Crefformat{lemma*}{#2Lemma~#1#3}

\newtheorem{proposition}[environment]{Proposition}
\crefname{proposition}{proposition}{propositions}
\crefformat{proposition}{#2Proposition~#1#3}
\Crefformat{proposition}{#2Proposition~#1#3}

\newtheorem{corollary}[environment]{Corollary}
\crefname{corollary}{corollary}{corollaries}
\crefformat{corollary}{#2Corollary~#1#3}
\Crefformat{corollary}{#2Corollary~#1#3}

\newtheorem{theorem}[environment]{Theorem}
\crefname{theorem}{theorem}{Theorems}
\crefformat{theorem}{#2Theorem~#1#3}
\Crefformat{theorem}{#2Theorem~#1#3}

\newtheorem*{theorem*}{Theorem}
\crefname{theorem*}{theorem}{Theorems}
\crefformat{theorem*}{#2Theorem~#1#3}
\Crefformat{theorem*}{#2Theorem~#1#3}

\newtheorem{conjecture}[environment]{Conjecture}
\crefname{conjecture}{conjecture}{Conjectures}
\crefformat{conjecture}{#2Conjecture~#1#3}
\Crefformat{conjecture}{#2Conjecture~#1#3}

\newtheorem*{hypothesis*}{Hypothesis}
\crefname{hypothesis*}{conjecture}{Conjectures}
\crefformat{hypothesis*}{#2Conjecture~#1#3}
\Crefformat{hypothesis*}{#2Conjecture~#1#3}

\crefname{observation}{observation}{Observations}
\crefformat{observation}{#2Observation~#1#3}
\Crefformat{observation}{#2Observation~#1#3}

\crefname{example}{example}{examples}
\crefformat{example}{#2Example~#1#3}
\Crefformat{example}{#2Example~#1#3}

\crefname{remark}{remark}{remarks}
\crefformat{remark}{#2Remark~#1#3}
\Crefformat{remark}{#2Remark~#1#3}

\crefname{figure}{figure}{figures}
\crefformat{figure}{#2Figure~#1#3}
\Crefformat{figure}{#2Figure~#1#3}

\crefname{equation}{equation}{Equations}
\crefformat{equation}{#2#1#3}
\Crefformat{equation}{#2#1#3}

\crefname{chapter}{chapter}{chapters}
\crefformat{chapter}{#2Chapter~#1#3}
\Crefformat{chapter}{#2Chapter~#1#3}

\crefname{section}{section}{sections}
\crefformat{section}{#2Section~#1#3}
\Crefformat{section}{#2Section~#1#3}

\crefname{algorithm}{algorithm}{algorithms}
\crefformat{algorithm}{#2Algorithm~#1#3}
\Crefformat{algorithm}{#2Algorithm~#1#3}

\crefname{notation}{notation}{notations}
\crefformat{notation}{#2Notation~#1#3}
\Crefformat{notation}{#2Notation~#1#3}

\crefname{question}{question}{questions}
\crefformat{question}{#2Question~#1#3}
\Crefformat{question}{#2Question~#1#3}

\crefname{problem}{problem}{problem}
\crefformat{problem}{#2problem~#1#3}
\Crefformat{problem}{#2problem~#1#3}

\newtheorem{claim}{Claim}[environment]
\crefname{claim}{claim}{claims}
\crefformat{claim}{#2Claim~#1#3}
\Crefformat{claim}{#2Claim~#1#3}

\crefname{definition}{definition}{definitions}
\crefformat{definition}{#2Definition~#1#3}
\Crefformat{definition}{#2Definition~#1#3}

\definecolor{CornflowerBlue}{rgb}{0.39, 0.58, 0.93}
\definecolor{Magenta}{rgb}{0.50, 0.0, 0.50}
\definecolor{AppleGreen}{rgb}{0.55, 0.71, 0.0}
\definecolor{AO}{rgb}{0.0, 0.5, 0.0}

\definecolor{DeepCarrotOrange}{rgb}{0.91, 0.41, 0.17}
\definecolor{BananaYellow}{rgb}{1.0, 0.88, 0.21}

\hypersetup{
colorlinks=true,
linkcolor=AO!65!black,
citecolor=AO!65!black,
urlcolor=AppleGreen!65!black,
bookmarksopen=true,
bookmarksnumbered,
bookmarksopenlevel=2,
bookmarksdepth=3
}

\title{Coloring Graphs With No Totally Odd Clique Immersion}
\author[1]{Caleb McFarland\footnote{Supported in part by the National Science Foundation under Grant No. DMS- 2452111.}
}
\affil[1]{School of Mathematics, Georgia Institute of Technology}
\date{August 2025; revised July 2026}

\begin{document}

\maketitle

\begin{abstract}
    We prove that graphs that do not contain a totally odd immersion of $K_t$ are $\mathcal{O}(t)$-colorable. In particular, we show that any graph with no totally odd immersion of $K_t$ is the union of a bipartite graph and a graph which forbids an immersion of $K_{\mathcal{O}(t)}$. Our results are algorithmic, and we give a fixed-parameter tractable algorithm (in $t$) to find such a decomposition.
\end{abstract}

\section{Introduction}
Hadwiger conjectured that every graph without a $K_t$-minor is $(t-1)$-colorable \cite{hadwiger1943klassifikation}. This would be a vast generalization of the famous Four-Color Theorem which to date only has a computer proof \cite{appel1989every, Robertson1997FourColorTheorem}. Despite a huge amount of recent progress on this problem \cite{delcourt2025reducing, kostochka1984lower, norin2023breaking, postle2020even, thomason1984extremal}, a linear bound remains out of reach. So, in an effort to understand the limits of Hadwiger's Conjecture, several other notions of graph containment have been considered. In particular, research in the area focuses on minors~\cite{delcourt2025reducing, hadwiger1943klassifikation, kostochka1984lower, norin2023breaking, thomason1984extremal}, subdivisions~\cite{bollobas1998proof, Catlin1979hajos, Hajos1961uber, komlos1996topological}, immersions~\cite{Abu2003graph, Devos2014Minimum, Gauthier2019Forcing, Lescure1988problem}, and their ``odd" variants~\cite{Churchley2017Odd, Jensen1995GraphColoring, Kawarabayashi2013Totally, Steiner2022Asymptotic}. The odd variants were introduced in an effort to capture the parity of cycles; a graph is bipartite if and only if it forbids $K_3$ as an odd minor/subdivision/immersion. (These three notions are equivalent for $K_3$ but not for larger cliques.)

It has been a general trend in this area that, after sometimes a great deal of work, it is discovered that coloring in the odd case reduces to coloring in the standard case up to multiplicative factors \cite{Kawarabayashi2013Totally, Steiner2022Asymptotic}. See \cref{fig:asymptotics_summary} for a summary of asymptotic coloring bounds in each case. Here, we show that the same trend holds for immersions.

A graph $G$ contains an \emph{immersion} of $K_t$ if there is a way of mapping the vertices of $K_t$ into distinct vertices of $G$ and the edges of $K_t$ into edge-disjoint trails in $G$ between the corresponding vertices of $G$. An immersion is said to be \emph{totally odd} if every trail has odd length. Several different definitions of immersions and ``odd'' immersions appear in the literature \cite{Churchley2017Odd, Jimenez2024totally}. The version we consider is natural from the perspective of edge-connectivity; its roots can be traced back to results of Lov\'{a}sz~\cite{LovaszBook79} and Mader~\cite{MaderMaintainConn} about splitting off while maintaining edge-connectivity. We refer the reader to \cref{section:Preliminaries} for formal definitions related to immersions. We are now ready to state our main result.

\begin{theorem}\label{thm:UnionOfBipAndNoImmersion}
    For every positive integer $t$, every graph that does not contain a totally odd immersion of $K_t$ is the union of a bipartite graph and a graph which does not contain $K_{98{\small,}000t + 4{\small,}410{\small,}071}$ as an immersion.
\end{theorem}
\noindent
We note that for coloring it suffices to consider simple graphs, but the graphs in \cref{thm:UnionOfBipAndNoImmersion} need not be simple.

It is known that graphs that do not contain $K_t$ as an immersion are $\mathcal{O}(t)$-colorable \cite{Devos2014Minimum}. The best bound of this form is due to Gauthier, Le, and Wollan~\cite{Gauthier2019Forcing}, who prove that every graph that does not contain $K_t$ as an immersion is $\lceil3.54t + 3\rceil$-colorable. Together with \cref{thm:UnionOfBipAndNoImmersion}, this immediately implies the first linear bound on the chromatic number of graphs that do not contain a totally odd $K_t$-immersion. The previous best known bound was $\mathcal{O}(t^2)$ \cite{echeverria2025topological, Kawarabayashi2013Totally}.

\begin{restatable}{corollary}{MainColoringTheorem}\label{thm:MainColoringResult}
    For every positive integer $t$, every graph that does not contain a totally odd immersion of $K_t$ is $(700{\small,}000t + 32{\small,}000{\small,}000)$-colorable.
\end{restatable}

Inspired by well-known conjectures about coloring graphs with no immersion of $K_t$, Churchley~\cite{Churchley2017Odd} conjectured that every graph that does not contain a totally odd immersion of $K_t$ is $(t-1)$-colorable. \cref{thm:MainColoringResult} resolves Churchley's conjecture up to a constant factor.

Churchley's conjecture has also been posed for totally odd strong immersions by Jim\'{e}nez, Quiroz, and Caro~\cite{Jimenez2024totally}. An immersion of $K_t$ is called \emph{strong} if the edges of $K_t$ correspond to edge-disjoint paths in $G$ (as opposed to trails), and the vertices of $K_t$ are mapped to vertices in $G$ which do not appear as internal vertices of these paths. The definition we give for immersions is sometimes called \emph{weak} immersions. Echeverr\'{i}a, Jim\'{e}nez, Mishra, Pastine, Quiroz, and Y\'{e}pez~\cite{echeverria2025topological} conjectured that every graph that does not contain a totally odd strong immersion of $K_t$ is $\mathcal{O}(t)$-colorable. Our techniques do not extend to strong immersions, and so we leave this conjecture as an open problem. We note that strong immersions, despite being more similar to subdivisions, are often less natural from a structural perspective. For instance, Robertson and Seymour~\cite{robertson2010graph} were able to extend their proof of Wagner's conjecture to show that graphs are well-quasi-ordered under the (weak) immersion relation. However, whether the same is true for strong immersions remains open \cite{liu2020recent, liu2023well}. This, along with their more natural description in terms of the ``splitting off'' operation, suggest that perhaps weak immersions are more analogous to minors than strong immersions.

The previous best bound on the chromatic number of graphs that do not contain a totally odd $K_t$-immersion was $79t^2/4$, given by Kawarabayashi~\cite{Kawarabayashi2013Totally} more generally for the case of totally odd subdivisions. It seems that an $\mathcal{O}(t^2)$ bound is a natural barrier to many approaches, as such a bound is implied by results of Kawarabayashi~\cite{Kawarabayashi2013Totally}, Churchley~\cite{Churchley2017Odd}, and McCarty, Wollan, and the author~\cite{McCartyMW2026Structure}. The main difficulty we overcome is constructing a totally odd $K_t$ immersion without working directly with a ``flower''. We refer the reader to \cref{subsection:comparison} for more details.

Finally, we mention that our proofs are algorithmic. That is, we obtain the following theorem.

 \begin{restatable}{theorem}{DecompositionAlgorithmThm}\label{thm:DecompositionAlgorithm}
     There exists an algorithm that takes as input a multigraph $G$ and a positive integer $t$ and finds either a totally odd $K_t$-immersion in $G$, or a bipartite subgraph $H$ of $G$ such that $G \setminus E(H)$ has no $K_{98{\small,}000t + 4{\small,}410{\small,}071}$-immersion. Furthermore, this algorithm runs in $\mathcal{O}_t(|V(G)|^4|E(G)| + |E(G)|^4)$ time.
 \end{restatable}

The above can be used, for instance, to give an algorithm for coloring a graph $G$ with at most $2\chi(G) + \mathcal{O}(t)$ colors if $G$ has no totally odd $K_t$-immersion by reducing to a result of Kawarabayashi and Kobayashi~\cite{kawarabayashi2012list}. This recovers an algorithm of Kawarabayashi, known more generally for graphs without a totally odd subdivision~\cite{Kawarabayashi2013Totally}.

\begin{figure}
    \centering
    \begin{tabular}{c|c|c|c|c|}
         & \multicolumn{2}{|c|}{standard} & \multicolumn{2}{c|}{odd}\\
         & Lower & Upper & Lower & Upper \\
         \hline
         subdivisions & $\Omega(t^2/\log t)$ & $\mathcal{O}(t^2)$ & $\Omega(t^2/\log t)$ & $\mathcal{O}(t^2)$\\
        \hline
        minors & $t-1$ & $\mathcal{O}(t\log\log t)$ & $(1.5 - o(1))t$ & $\mathcal{O}(t\log\log t)$\\
        \hline
        immersions & $t-1$ & $\mathcal{O}(t)$ & $t-1$ & \Cancel[red]{$\mathcal{O}(t^2)$}\textcolor{red}{$\mathcal{O}(t)$}
    \end{tabular}
    \caption{Lower and upper bounds for the maximum chromatic number of graphs without a (standard/odd) $K_t$ subdivision/minor/immersion. Our improvement is noted in red.}
    \label{fig:asymptotics_summary}
\end{figure}

\subsection{Related work}

To put our results in context we now give an overview of related work on coloring in (odd) \linebreak
minor/immersion/subdivision free graphs. Our results are motivated by the following well-known conjecture of Abu-Khzam and Langston~\cite{Abu2003graph} (also given by Lescure and Meyniel~\cite{Lescure1988problem} for strong immersions).

\begin{conjecture}[{\cite{Abu2003graph}}]\label{conj:ImmersionHadwigers}
    For every positive integer $t$, every graph that does not contain a $K_t$-immersion is $(t-1)$-colorable.
\end{conjecture}

\cref{conj:ImmersionHadwigers} is known for $t \leq 7$ due to work of Lescure and Meyniel~\cite{Lescure1988problem} and DeVos, Kawarabayashi, Mohar, and Okamura~\cite{DeVos2010Immersing}. \cref{conj:ImmersionHadwigers} is also known up to a constant factor, as first shown by DeVos, Dvo{\v{r}}{\'a}k, Fox, McDonald, Mohar, Bojan, and Scheide~\cite{Devos2014Minimum}, who proved that graphs with no $K_t$-immersion are $200t$-colorable. The best known asymptotic bound is that of Gauthier, Le, and Wollan~\cite{Gauthier2019Forcing} given below.

\begin{theorem}[{\cite[Theorem~1.6]{Gauthier2019Forcing}}]\label{thm:coloringNoImmersion}
    For every positive integer $t$, every graph that does not contain a $K_t$-immersion is $\lceil3.54t + 3\rceil$-colorable.
\end{theorem}

In terms of recent progress, Fox, Pach, and Suk~\cite{fox2025immersions} showed that \cref{conj:ImmersionHadwigers} holds for graphs with at most $(1.64 - o(1))t$ vertices. This is in contrast to Hadwiger's Conjecture, where coloring small graphs seems like a bottleneck to improving current approaches~\cite{delcourt2025reducing}. Liu, Wang, and Yang~\cite{liu2022clique} showed that \cref{conj:ImmersionHadwigers} holds asymptotically for graphs without a fixed complete bipartite subgraph. In a slightly different direction, Kawarabayashi and Kobayashi~\cite{kawarabayashi2012list} gave an additive approximation algorithm for (list-)coloring graphs with no clique immersion. They also gave an Erd\H{o}s-P\'{o}sa property for $K_t$-immersions in $1.5t$-edge-connected graphs. Kakimura and Kawarabayashi~\cite{kakimura2016coloring} gave an FPT algorithm to decided whether a graph with no $K_t$-immersion is $(t-2)$-colorable. \cref{conj:ImmersionHadwigers} has also been studied under different coloring notions, such as clustered coloring~\cite{liu2023immersion} and nonrepetitive coloring~\cite{wollan2019nonrepetitive}.

\cref{conj:ImmersionHadwigers} has obvious similarities with the famous Hadwiger's Conjecture, which says that for every integer $t \geq 1$, every graph that does not contain a $K_{t}$-minor is $(t-1)$-colorable. It is still open whether graphs with no $K_t$-minor are $\mathcal{O}(t)$-colorable. There has been a long line of work improving the asymptotic bounds for Hadwiger's Conjecture \cite{kostochka1984lower, norin2023breaking, postle2020even, thomason1984extremal}. The best bound currently is that of Delcourt and Postle~\cite{delcourt2025reducing}, who proved that graphs with no $K_t$-minor are $\mathcal{O}(t\log\log t)$-colorable.

\cref{conj:ImmersionHadwigers} and Hadwiger's Conjecture, while incomparable, are both weakenings of a conjecture of Haj\'{o}s~\cite{Hajos1961uber}, who conjectured that for every positive integer $t$, every graph that does not contain a subdivision of $K_{t}$ is $(t-1)$-colorable. This was disproven by Catlin~\cite{Catlin1979hajos} for all $t \geq 7$. It is not hard to show that Haj\'{o}s conjecture holds for $t \leq 4$, but the case of $t \in \{5,6\}$ remains open (see \cite{he2020kelmans, xie20224} for recent progress). Bollob\'{a}s and Thomason~\cite{bollobas1998proof}, and independently Koml\'{o}s and Szemer\'{e}di~\cite{komlos1996topological}, showed that graphs that do not contain a $K_t$-subdivision are $\mathcal{O}(t^2)$-colorable. Erd\H{o}s and Fajtlowicz~\cite{erdHos1981conjecture} proved that Haj\'{o}s' conjecture is false for almost all graphs. In fact, there are graphs with no $K_t$-subdivision and chromatic number $\Omega(t^2/\log t)$.

For $t = 3$, both Hadwiger's Conjecture and \cref{conj:ImmersionHadwigers} are equivalent to the statement that trees are 2-colorable. There have been different attempts to strengthen both conjectures to give a more precise classification of the 2-colorable case. For Hadwiger's Conjecture, Gerards and Seymour~\cite{Jensen1995GraphColoring} proposed the following odd-minor variant.

\begin{conjecture}[{\cite{Jensen1995GraphColoring}}]\label{conj:oddHadwigers}
    For every positive integer $t$, every graph that does not contain $K_{t}$ as an odd-minor is $(t-1)$-colorable.
\end{conjecture}

A graph is an \emph{odd-minor} of another if it can be formed from a subgraph via contracting edge cuts. Odd-minors are important because they preserve the parity of cycles. \cref{conj:oddHadwigers} was recently shown to be false by K{\"u}hn, Sauermann, Steiner, and Wigderson~\cite{kuhn2025disproof}. They showed that there exist graphs that do not contain $K_t$ as an odd-minor and have chromatic number at least $(1.5 - o(1))t$. Churchley~\cite{Churchley2017Odd} conjectured the following for immersions.

\begin{conjecture}[{\cite{Churchley2017Odd}}]\label{conj:OddImmersionHadwigers}
    For every positive integer $t$, every graph that does not admit a totally odd immersion of $K_{t}$ is $(t-1)$-colorable.
\end{conjecture}

In both of the above conjectures, the $t=3$ case now correctly identifies 2-colorable graphs as those with no odd cycles. The conjectures also say something about dense graphs, as graphs forbidding an odd-minor or totally odd immersion need not be sparse, unlike graphs which forbid a minor \cite{kostochka1984lower, thomason1984extremal} or immersion \cite{Devos2014Minimum}. \cref{conj:OddImmersionHadwigers} is known for $t=4$ via a stronger result of Zang~\cite{Zang1998proof} and independently Thomassen~\cite{Thomassen2001totally} which replaces ``totally odd immersion'' with ``totally odd subdivision''. Jim\'{e}nez, Quiroz, and Caro~\cite{Jimenez2024totally} proved \cref{conj:OddImmersionHadwigers} for line graphs of multigraphs where every edge has the same multiplicity. Echeverr\'{i}a, Jim\'{e}nez, Mishra, Quiroz, and Y\'{e}pez~\cite{echeverria2025totally} showed that \cref{conj:OddImmersionHadwigers} holds for certain graph products. 

\cref{conj:OddImmersionHadwigers} implies in particular that every $n$-vertex graph $G$ contains a totally odd $K_{\lceil n/\alpha(G)\rceil}$-immersion (we write $\alpha(G)$ for the maximum size of an independent set in $G$). Just as in the case of Hadwiger's conjecture, the $\alpha(G) = 2$ case is of particular interest. Gauthier, Le, and Wollan~\cite{Gauthier2019Forcing} showed that every $n$-vertex graph $G$ with $\alpha(G) = 2$ contains an $K_{2\lfloor n/5\rfloor}$-immersion, and in fact their proof gives a totally odd immersion. Bustamante, Quiroz, Stein, and Zamora~\cite{bustamante2022clique} extend this to $\alpha(G) \geq 3$ by showing that every n-vertex graph $G$ contains a totally odd $K_t$-immersion for $t \geq \lfloor n/(2.25\alpha(G) - 2.25)\rfloor$. Very recently, Echeverr\'{i}a and McDonald~\cite{echeverria2026coloring} showed that every graph $G$ with $\alpha(G) = 2$ that does not contain a totally odd $K_t$-immersion is $\lceil 1.5(t-1)\rceil$-colorable.

Steiner~\cite{Steiner2022Asymptotic} answered \cref{conj:oddHadwigers} up to constant factor by showing that the question is equivalent to Hadwiger's Conjecture up to a factor of 2. This greatly simplified previous work on the asymptotics of \cref{conj:oddHadwigers} \cite{delcourt2025reducing, Geelen2009oddMinor, kawarabayashi2009note, norin2022new, postle2020further}. In particular, Steiner showed that in a graph with no odd $K_t$-minor, there exist disjoint bipartite induced subgraphs which can be identified to form a graph with no $K_t$-minor. This implies that if $K_t$-minor free graphs are $f(t)$-colorable, then odd $K_t$-minor free graphs are $2f(t)$-colorable. If Hadwiger's Conjecture is true, this implies that the maximum chromatic number of $K_t$ odd-minor free graphs is between $(1.5 - o(1))t$ and $2t$.

Similarly, Kawarabayashi~\cite{Kawarabayashi2013Totally} showed that graphs that do not contain a totally odd subdivision of $K_t$ are $\mathcal{O}(t^2)$-colorable by reducing to the case without parity restrictions. A subdivision is called \emph{totally odd} if every edge is replaced by a path of odd length. This improved on the exponential upper bound given by Thomassen~\cite{thomassen1983graph}. Kawarabayashi's result gave the previously best-known upper bound on the chromatic number of graphs that do not contain a totally odd $K_t$-immersion. We continue the trend of Kawarabayashi and Steiner by showing that the upper bound on the chromatic number of graphs that do not contain a totally odd $K_t$-immersion is the same as that of graphs with no $K_t$-immersion up to a constant factor.

 Graphs which forbid an immersion not only have linear chromatic number, they also have linear degeneracy. A graph is \emph{$k$-degenerate} if every subgraph has minimum degree at most $k$. The linear degeneracy of $K_t$-immersion free graphs was first shown by DeVos et al.~\cite{Devos2014Minimum}, with the best bound being the following due to Gauthier, Le, and Wollan~\cite{Gauthier2019Forcing}. This result will be very helpful in constructing immersions from dense graphs.

\begin{theorem}[{\cite[Theorem~1.4]{Gauthier2019Forcing}}]\label{thm:MinDegImmersion}
    If $G$ is a simple graph with minimum degree at least $7t+7$, then $G$ has a $K_t$-immersion.
\end{theorem}

This leads to the following corollary of \cref{thm:UnionOfBipAndNoImmersion}.

\begin{corollary}
    For every positive integer $t$, every simple graph that does not admit a totally odd immersion of $K_t$ is the union of a bipartite graph and a $(686{\small,}000t + 31{\small,}000{\small,}000)$-degenerate graph.
\end{corollary}

The bounds in \cref{thm:MainColoringResult} and \cref{thm:UnionOfBipAndNoImmersion} are not optimal even for the given proof strategy, but we prioritize readability over obtaining the smallest possible bound. We also note that any improvements to the bounds in \cref{thm:coloringNoImmersion} and \cref{thm:MinDegImmersion} lead to immediate improvements in our bounds. The connection to \cref{thm:MinDegImmersion} will become clear during the proof of \cref{thm:UnionOfBipAndNoImmersion} (see \cref{lem:MakingTotallyEven} and \cref{thm:makingImmersionOdd}).

\subsection{Comparison to previous approaches}\label{subsection:comparison}

One of the main tools used in this work, \cref{cor:OddCircuitsEP}, also appears in other work on immersions with group constraints (see for instance \cite[Corollary 3.2]{McCartyMW2026Structure}). However, the work of McCarty, Wollan, and the author~\cite{McCartyMW2026Structure} and other work on structure theorems for graphs forbidding such immersions \cite{Churchley2017Odd, WOLLAN2015} work directly with \emph{flower immersions}. The \emph{$(k,n)$-flower} is the graph obtained from $K_{1,n}$ by making each edge have multiplicity $k$. It is not hard to see that the $(t,t)$-flower contains $K_t$ as an immersion by splitting off at the high degree vertex. Moreover, if we could control the parity of the trails in the flower immersion, then we could also find $K_t$ as a totally odd immersion.

Previous approaches first found a flower immersion where all the trails are even (such an immersion is called \emph{totally even}). Then, because flowers contain a single vertex incident to every edge, they used a version of \cref{cor:OddCircuitsEP} to find many edge-disjoint odd circuits which are ``rooted'' at that vertex of the flower. The heart of the argument is then to ``disentangle'' the edge-disjoint odd circuits from the totally even flower immersion. The main difficulty in this paper is creating a totally odd clique immersion from a totally even clique immersion and many edge-disjoint odd circuits. This step is more difficult because we lack the special vertex incident to every edge. Previous approaches can never get the desired coloring bound exactly because of this high degree vertex; graphs of maximum degree less than $t^2$ can not contain a $(t,t)$-flower immersion.

\section{Preliminaries}\label{section:Preliminaries}

Graphs may have loops and parallel edges unless they are explicitly stated to be simple. An \emph{edge block} of a graph is a maximal bridgeless subgraph. For $X \subseteq V(G)$, we denote by $G \setminus X$ the graph formed from $G$ by deleting the vertices in $X$. For $F \subseteq E(G)$, we denote by $G \setminus F$ the graph formed from $G$ by deleting the edges in $F$. We denote by $G[X]$ and $G[F]$ the graph induced by $X \subseteq V(G)$ or $F \subseteq E(G)$ respectively. We denote by $G_1 \cup G_2$ the graph with vertex set $V(G_1) \cup V(G_2)$ and edge set $E(G_1) \cup E(G_2)$.

A \emph{trail} $T$ in a graph $G$ is a sequence of distinct edges $e_1e_2\dots e_n$ which can be oriented so that the head of $e_i$ is the tail of $e_{i+1}$ for all $i \in [n-1]$. We call such an orientation the orientation of the edges \emph{in $T$}. We call the tail of $e_1$ in $T$ and the head of $e_n$ in $T$ the \emph{ends} of $T$. We call the ends of $e_i$ for $i \in \{2, \dots, n-1\}$ the \emph{internal vertices} of the trail. Note that a vertex may both be an end and an internal vertex of a trail. We say a trail \emph{hits} a vertex if that vertex is an end of some edge in the trail. We denote by $T^{-1}$ the trail $e_ne_{n-1}\ldots e_1$. A \emph{subtrail} of a trail $T$ is a consecutive subsequence of the edges in $T$. A \emph{circuit} is a trail where both ends are identical.

We say that a graph $G$ admits a \emph{(weak) immersion} of a graph $H$ if there exist injective functions $\varphi_V: V(H) \rightarrow V(G)$ and $\varphi_E: E(H) \rightarrow \{\text{trails in }G\}$ such that 
\begin{enumerate}
    \item for all $e \in E(H)$ with ends $u,v \in V(H)$, $\varphi_E(e)$ has ends $\varphi_V(u), \varphi_V(v)$, and
    \item for all distinct $e_1, e_2 \in E(H)$, $\varphi_E(e_1)$ and $\varphi_E(e_2)$ are edge-disjoint.
\end{enumerate}
We call $\{\varphi_V(v) : v \in V(H)\}$ the \emph{branch vertices} of the immersion, and we call $\{\varphi_E(e) : e \in E(H)\}$ the \emph{branch trails}. We say that $G$ admits a \emph{strong immersion} of $H$ if the branch trails of the immersion have no internal vertex in the branch vertices of the immersion.

An immersion $H$ in a graph $G$ is said to be \emph{totally odd} (\emph{totally even}) if all of the branch trails have odd (even) length. We note that in the literature it is common to replace ``trails'' in the definition with ``paths''. The resulting definition is equivalent for immersions, but no longer equivalent for totally odd (weak or strong) immersions. To see this consider a branch trail which is the union of an even path and an odd cycle. We follow the definition as given in Churchley's original conjecture~\cite{Churchley2017Odd}, which is motivated by the following equivalent definition of immersions.

Equivalently, a graph $G$ admits an immersion of a graph $H$ if $H$ can be formed by deleting vertices, deleting edges, and splitting off pairs of incident edges. Given two incident edges $e,f$ such that $e$ has ends $u,v$ and $f$ has ends $v,w$, the operation of \emph{splitting off the edges $e,f$ at $v$} is done by deleting $e$ and $f$ and adding an edge with ends $u,w$. Then an edge $e$ in the immersion $H$ correspond to a trail in the original graph $G$ with edge set equal to the edges which were split off to form $e$.

We also introduce the notion of \emph{transitions}. Let $T = e_1e_2\dots e_n$ be a trail. The \emph{transitions} of $T$ are the ordered pairs $(\{e_i, e_{i+1}\}, v)$ where $\{e_i, e_{i+1}\}$ are edges that appear in sequence along the trail with $v$ equal to both the head of $e_i$ in $T$ and the tail of $e_{i+1}$ in $T$. Note that $T^{-1}$ has the same set of transitions as $T$. We say a transition $(\{e_i, e_{i+1}\}, v)$ is a transition \emph{at $v$}, and that the transition \emph{contains} $e_i, e_{i+1}$. The \emph{transitions} of an immersion are the union of the transitions of its branch trails. The transitions of an immersion encode splitting off operations needed to form the immersion.

A main ingredient of our proof is an Erd\"{o}s-P\'{o}sa property for odd circuits. We derive the result as a corollary of the following result on odd $A$-paths due to Geelen, Gerards, Reed, Seymour, and Vetta~\cite{Geelen2009oddMinor}. An \emph{$A$-path} is a path which has both endpoints in a set $A \subseteq V(G)$ and is internally vertex-disjoint from $A$. An \emph{$x$-circuit} is a circuit hitting a vertex $x \in V(G)$ with no internal vertex equal to $x$.

\begin{lemma}[{\cite[Lemma~11]{Geelen2009oddMinor}}]\label{lem:OddPathsEP}
    Let $G$ be a graph and $k \geq 1$ be an integer. For any set $A \subseteq V(G)$, either
    \begin{enumerate}
        \item there are $k$ vertex-disjoint odd $A$-paths, or
        \item there is a set $X \subseteq V(G)$ of size at most $2k-2$ such that every $A \setminus X$ path in $G\setminus X$ has even length.
    \end{enumerate}
\end{lemma}

This gives the following corollary by applying the above to the line graph of a 1-subdivision.

\begin{corollary}\label{cor:OddCircuitsEP}
    Let $G$ be a graph and $k \geq 1$ be an integer. For any vertex $x \in V(G)$, either
    \begin{enumerate}
        \item there are $k$ edge-disjoint odd $x$-circuits, or
        \item there is a set $X \subseteq E(G)$ of size at most $2k-2$ such that every $x$-circuit in $G \setminus X$ has even length.
    \end{enumerate}
\end{corollary}
\begin{proof}
    Let $G'$ be the $1$-subdivision of $G$, and let $H$ be the line graph of $G'$. Note that each vertex $v \in V(G)$ corresponds to a clique $\delta_{G'}(v)$ in $H$, and that these cliques are pairwise vertex-disjoint. For an edge $e \in E(G)$ with ends $u,v \in V(G)$, the vertex that subdivides $e$ corresponds to a clique of size two in $H$, and the edge $f_e$ of this clique is exactly between the cliques $\delta_{G'}(u), \delta_{G'}(v)$. Let $H'$ be formed by subdividing every edge in a clique of type $\delta_{G'}(v)$ for $v \in V(G)$, but not the edges $f_e$ for $e \in E(G)$. Note that there exists a 2-coloring of $H'$ such that the edges $f_e$ for $e \in E(G)$ are exactly the monochromatic edges.

    There is a natural mapping from vertex-disjoint odd cycles in $H'$ to edge-disjoint odd circuits in $G$ by mapping a cycle in $H'$ to the edges $e\in E(G)$ so that $f_e$ is contained in the cycle. Note that if a cycle in $H'$ has odd length, then it must contain an odd number of monochromatic edges in any 2-coloring, and so must use an odd number of edges $f_e$. In particular the odd cycle cannot be entirely contained in the subdivision of a clique $\delta_{G'}(v)$, and thus the resulting edges form a circuit in $G$. Let $A \coloneq \delta_{G'}(x) \subseteq V(H')$. Note that every odd $A$-path in $H'$ can be extended to an odd cycle in $H'$ by taking the path of length 2 between its endpoints given by the subdivided edge of the clique $\delta_{G'}(x)$ in $H$. Thus if $H'$ has $k$ vertex-disjoint odd $A$-paths, then $G$ has $k$ edge-disjoint odd $x$-circuits. 

    So by \cref{lem:OddPathsEP}, we may assume that there is a set $X \subseteq V(H')$ of size at most $2k-2$ which hits all odd $A$-paths in $H'$. We may assume that $X$ is a subset of $V(H)$ since vertices in $V(H')\setminus V(H)$ have degree two in $H'$, and we could have included one of their neighbors in $X$ instead. Thus we can consider $X$ as a subset of $E(G')$. Then the set of edges in $G$ for which one of their two corresponding edges in $G'$ lies in $X$ hits all odd $x$-circuits of $G$. The result follows.
\end{proof}

Note that if every $x$-circuit of $G \setminus X$ has even length, then the edge block containing $x$ in $G \setminus X$ must be bipartite. To see this, note that by taking two edge-disjoint paths from $x$ to an odd cycle we can always obtain an odd $x$-circuit.

We note that both \cref{lem:OddPathsEP} and \cref{cor:OddCircuitsEP} are specific cases of a more general result about group-labeled graphs; see \cite{chudnovsky2006packing} and \cite{McCartyMW2026Structure} respectively.

\section{From totally even to totally odd}
In this section we first show that every large clique immersion contains a smaller one which is totally even. We then show, essentially, that if there are many edge-disjoint odd circuits or trails with ends in the branch vertices of a totally even clique immersion, then there is a totally odd $K_t$-immersion. In graphs which forbid a totally odd $K_t$-immersion, this will allow us to apply \cref{cor:OddCircuitsEP} in order to separate off a bipartite piece of our graph in the presence of a large clique immersion.

\begin{lemma}\label{lem:MakingTotallyEven}
    Let $G$ be a graph with a $K_{70t + 71}$-immersion. Then $G$ contains a totally even $K_t$-immersion. Moreover, there exists an algorithm that takes as input a graph $G$ and the branch trails of a $K_{70t + 71}$ immersion and returns the branch trails of a totally even $K_t$-immersion in $\mathcal{O}_t(|E(G)|)$ time.
\end{lemma}
\begin{proof}
    We split into cases based on the proportion of branch trails of odd length. If at least $4/5$ of the trails have odd length, then we consider the totally odd immersion $G'$ of $G$ formed from the $K_{70t + 71}$-immersion by restricting to the trails of odd length. Note that $G'$ has average degree at least $(70t + 71-1)(4/5) =56t + 56$. We then find a bipartite subgraph of $G'$ of average degree at least $28t + 28$, which itself has a subgraph $H$ of minimum degree at least $14t + 14$. By \cref{thm:MinDegImmersion}, $H$ has a $K_{2t}$-immersion. By keeping the branch vertices on only one side of the bipartition we obtain the desired totally even $K_t$-immersion.

    If at least $1/5$ of the branch trails have even length, we instead consider the totally even immersion $G'$ formed from the $K_{70t + 71}$-immersion by restricting to the trails of even length. Note that $G'$ has average degree at least $14t + 14$. So $G'$ has a subgraph $H$ of minimum degree at least $7t + 7$, which by \cref{thm:MinDegImmersion} contains the desired totally even $K_t$-immersion.

    For the algorithm, note that after finding the parity of each branch trail in $\mathcal{O}_t(|E(G)|)$ time, it suffices to work over a subgraph of $K_{70t + 71}$. We may then find a totally even $K_t$-immersion by brute force.
\end{proof}

We now show how to get a totally odd clique immersion from a totally even one when there are many edge-disjoint odd trails with ends in the branch vertices of the totally even clique immersion where each branch vertex is not the end of too many odd trails. We will first make our odd trails edge-disjoint from all but a few of the branch trails of the totally even immersion. We then delete the branch trails which are not disjoint, and we delete a few vertices to keep high minimum degree in the resulting immersion. We then make all of the odd trails into odd circuits, and group all of the vertices with many odd circuits together. We use the edges between these vertices and the rest of the graph to create a totally even clique immersion with many odd circuits at each branch vertex, and then we use those odd circuits to make the clique immersion totally odd. 

\begin{theorem}\label{thm:makingImmersionOdd}
    Let $t \geq 49$, and let $G$ be a graph with a totally even $K_{1400t}$-immersion with branch vertices $S \subseteq V(G)$. Consider constructing an auxiliary graph by adding a vertex $x$ adjacent to every vertex in $S$ with multiplicity $20t$. Suppose there is a collection of $3500t^2$ edge-disjoint odd $x$-circuits in this auxiliary graph. Then $G$ contains a totally odd $K_t$-immersion.

    Furthermore, there exists an algorithm that takes as input the above auxiliary graph, the branch trails of a totally even $K_{1400t}$-immersion, and the collection of $3500t^2$ edge-disjoint odd $x$-circuits and returns the branch trails of a totally odd $K_t$-immersion in $\mathcal{O}_t(|E(G)|^2)$ time.
\end{theorem}
\begin{proof}
    By removing the first and last edge of every odd $x$-circuit, every odd $x$-circuit in the auxiliary graph gives an odd trail in $G$ with ends in $S$. Note that this trail may in fact be a circuit, and both ends may be the same vertex in $S$. We call this set of odd trails $\mathcal{C}$. We call a vertex $v \in S$ \emph{full} if $v$ occurs as the end of $20t$ trails in $\mathcal{C}$, where we count a circuit in $\mathcal{C}$ with $v$ as both ends twice. Note that there are at most $3500t^2/20t=175t$ full vertices. We first choose $\mathcal{C}$ to be, informally, as disjoint from the branch trails of the $K_{1400t}$-immersion as possible.

    \begin{claim}\label{claim:DisentanglingCircuits}
        We can choose $\mathcal{C}$ such that for every branch trail of the $K_{1400t}$-immersion, either
        \begin{enumerate}
            \item it is edge-disjoint from every trail in $\mathcal{C}$,
            \item\label{itm:disentanglingOption2} the first or last edge of the branch trail is the first or last edge of a trail in $\mathcal{C}$, or
            \item both ends of the branch trail are full.
        \end{enumerate}
    \end{claim}
    \begin{proof}
        We choose $\mathcal{C}$ to minimize the number of transitions of trails in $\mathcal{C}$ which are not transitions of the $K_{1400t}$-immersion. Subject to this, we minimize the number of branch trails of the $K_{1400t}$-immersion for which \cref{itm:disentanglingOption2} does not hold. Suppose towards a contradiction that there exists a branch trail $T$ of the $K_{1400t}$-immersion such that none of the items hold. Let $v \in S$ be an end of $T$ which is not full, and let $e \in E(T)$ be the first edge along $T$ starting from $v$ which appears in some $C \in \mathcal{C}$.

        Suppose that $e$ is the first edge of $T$ starting from $v$. So by assumption, $e$ is neither the first nor last edge of $C$. By possibly replacing $C$ with $C^{-1}$, we may assume that $v$ is the tail of $e$ in $C$. Let $C_1$ and $C_2$ be non-empty trails so that $C = (C_1, C_2)$ and $e$ is the first edge of $C_2$. Let $f$ denote the last edge of $C_1$. One of $C_1, C_2$ has odd length, and both do not contain the transition $(\{f, e\}, v)$. Since $(\{f, e\}, v)$ is not a transition of the $K_{1400t}$-immersion, replacing $C$ with one of $C_1, C_2$ contradicts the choice of $\mathcal{C}$.
        
        Suppose instead that $e$ is not the first edge of $T$. By possibly replacing $C$ with $C^{-1}$, we may assume that $e$ has the same orientation in $T$ and $C$. Let $T_1, T_2, C_1, C_2$ be trails so that $T = (T_1, T_2)$,  $C = (C_1, C_2)$, and $e$ is the first edge in $T_2$ and $C_2$. The sum of the length of the trails $(T_1, C_2)$ and $(C_1, T_1^{-1})$ is odd, so one of these is odd. First suppose that $(T_1, C_2)$ is odd. If $C_1$ contains at least one edge, then the trail $(T_1, C_2)$ has fewer transitions than $C$ outside of the transitions of the $K_{1400t}$-immersion because it does not have the transition containing $e$ and its preceding edge in $C$. So suppose that $C_1$ contains no edges and $C$ starts with $e$. Note that replacing $C$ with $(T_1, C_2)$ does not increase the number of transitions of trails in $\mathcal{C}$ which are not transitions of the $K_{1400t}$-immersion. Furthermore, this replacement makes $T$ satisfy \cref{itm:disentanglingOption2} while not changing whether any other branch trail satisfies \cref{itm:disentanglingOption2}.
        
        For the final case, suppose that the trail $(C_1, T_1^{-1})$ is odd. If $C_2$ is not a subtrail of $T$, then $C_2$ has a transition outside those of the $K_{1400t}$-immersion, that being the transition including the first edge after $e$ in $C$ which is not an edge of $T$. So we may assume that $C_2$ is a subtrail of $T$. Note that replacing $C$ with $(C_1, T_1^{-1})$ does not increase the number of transitions of trails in $\mathcal{C}$ which are not transitions of the $K_{1400t}$-immersion. Furthermore, this replacement makes $T$ satisfy \cref{itm:disentanglingOption2} while not changing whether any other branch trail satisfies \cref{itm:disentanglingOption2}. Therefore, by replacing $C$ with one of $(T_1, C_2)$, $(C_1, T_1^{-1})$ we obtain a collection which contradicts our choice of $\mathcal{C}$.

        We can improve our collection $\mathcal{C}$ in $\mathcal{O}_t(|E(G)|)$ time, and there are at most $\mathcal{O}_t(|E(G)|)$ transitions of trails in $\mathcal{C}$, so we can update our collection $\mathcal{C}$ in $\mathcal{O}_t(|E(G)|^2)$ time.
    \end{proof}
    
    By \cref{claim:DisentanglingCircuits}, at most
    $$2(3500t) + \binom{175t}{2} < 22500t^2$$
    branch trails of the $K_{1400t}$-immersion share an edge with an element of $\mathcal{C}$. Let $G'$ be an immersion of $G$ formed from $K_{1400t}$ by deleting all edges whose branch trails share an edge with an element of $\mathcal{C}$. We may also split off edges to make each element in $\mathcal{C}$ into an edge or loop with end(s) in $S$, and we remember that these correspond to trails of odd length. Note that the edges and loops of $\mathcal{C}$ are not in $G'$. Algorithmically, we may now work over $G'$ and these edges or loops in $\mathcal{C}$, so the rest of the proof has a runtime depending only on $t$.

    We now delete all the $100t$ vertices in $G'$ of lowest degree. We claim that all remaining vertices in $G'$ had degree at least $950t$ in $G'$. Otherwise more than $100t$ vertices had their degree decreased by at least $450t$ compared to $K_{1400t}$, so more than $22500t^2$ edges were deleted to form $G'$ from $K_{1400t}$, a contradiction. After deleting these vertices we lose at most $(20t)(100t) = 2000t^2$ elements of $\mathcal{C}$, so we still have a collection $\mathcal{C}' \subseteq \mathcal{C}$ of size $1500t^2$. Because we deleted $100t$ vertices, the remaining vertices of $G'$ have minimum degree at least $850t$. Let $S' \subseteq S$ be these remaining vertices, and consider the subgraph of $G'$ induced on $S'$.

    We now make every element of $\mathcal{C}'$ into a loop with its end in $S'$ while sacrificing the minimum degree slightly. We do this by performing the following algorithm. For each non-loop edge in $\mathcal{C}'$ with ends $u,v \in S'$, we find the vertex $w$ in the common neighborhood of $u$ and $v$ of largest degree and split off the edges $uw$ and $wv$ to make the edge into a loop at either $u$ or $v$. Because the edge in $\mathcal{C}'$ corresponds to a trail of odd length and $G'$ is a totally even immersion, the created loop corresponds to an odd circuit in $G$. We claim this is always possible, and after doing this for all non-loop edges of $\mathcal{C}'$ we have minimum degree at least $800t$. We call the creation of a single loop a \emph{round}, and we now prove by induction on the number of rounds that we always have minimum degree at least $800t$. This implies that the common neighborhood of every pair of vertices is always non-empty since $|S'|=1300t$.
    
    Assume that in this round and every prior round, every vertex has minimum degree at least $800t$. Let $u$ and $v$ be the ends of an edge in $\mathcal{C}'$ which is turned into a loop in the next round. We first claim that there exists a vertex in the common neighborhood of $u$ and $v$ with degree at least $820t+2$. Since $u$ and $v$ have degree at least $800t$ by the inductive assumption and since $|S'| = 1300t$, the common neighborhood of $u$ and $v$ has size at least $300t$. If every vertex in the common neighborhood of $u$ and $v$ has degree at most $820t+1$, then at least $(30t-1)(300t)/2=4500t^2-150t$ edges have been removed $G'[S']$. However, in each round we remove two edges, and there are at most $1500t^2$ rounds, so at most $3000t^2$ edges have been removed, a contradiction. It follows that if $w$ is a vertex in the common neighborhood of $u$ and $v$ with largest degree, then $w$ still has degree at least $820t$ even after we split off the edges $uw$ and $wv$.

    It just remains to consider the degree of $u$ and $v$. By symmetry between $u$ and $v$, it suffices to consider the degree of $v$. If $v$ has degree at least $800t+1$, then we are done. Consider the first round before this one where $v$ has degree at most $820t+1$. Note that $v$ has degree at least $820t$ in that round since the degree goes down by at most $2$ per round. As we argued in the previous paragraph, from that round until this one, $v$ is never chosen as the vertex $w'$ of largest degree in the common neighborhood of any pair $u'$ and $v'$ joined by an edge in $\mathcal{C}'$. Moreover, the degree of $v$ can be decreased while being the end of an edge in $\mathcal{C}'$ at most $20t-1$ times strictly before this round. So indeed, in the current round $v$ has degree at least $800t+1$, as desired.

    Let $G''$ be the subgraph of $G'$ obtained by restricting to the vertex-set $S'$ and then performing the algorithm described above. Thus $G''$ has minimum degree at least $800t$. We also now have that every element of $\mathcal{C}'$ is a loop corresponding to an odd circuit in $G$, every loop is on a vertex in $S'$, and there are at most $20t$ loops at each vertex in $S'$.

    Let $A \subseteq S'$ be the $42t$ vertices with the largest amount of loops. We claim that each vertex in $A$ has at least $t/2$ loops. Otherwise, the number of loops is at most 
    $$(42t)(20t)+(1300t-42t)(t/2)=1469t^2<1500t^2=|\mathcal{C}'|,$$
    a contradiction. We now consider the bipartite graph $H$ which is the subgraph of $G''$ obtained by only keeping edges between $A$ and $S' \setminus A$. Note that $|E(H)| \geq (800t)(42t) - \binom{42t}{2} \geq 32500t^2$. Thus $H$ has average degree at least $(2 \cdot 32500t^2)/(1300t) = 50t$. We then find a subgraph $H'$ of $H$ of minimum degree at least $25t$, and let $A',B$ be the bipartition of $H'$ with $A' \subseteq A$. Note that $|B| \geq 25t$.

    We now show that $H'$ contains a $K_t$-immersion with branch vertices contained in $A'$. Consider ordering $B = \{v_1, \ldots, v_{\ell}\}$. We do the following for each $v_i$ in order from $i=1$ to $i=\ell$. If the subgraph induced on $N(v_i)$ contains an antimatching of size at least $12t$, we split off edges at $v_i$ to create the matching. An \emph{antimatching} of size $k$ is a matching in the complement graph. If there is no antimatching of size $12t$, then there exists a clique of size at least $t$ in $N(v_i)$ by greedily constructing an antimatching. To see this, note that for any maximal antimatching in the subgraph induced on $N(v_i)$, the remaining vertices in $N(v_i)$ form a clique. Furthermore, there are at least $25t-2(12t)=t$ vertices remaining. If there is always the desired antimatching, then after doing this for every vertex in $B$ we have added $12t\ell \geq (12t)(25t) = 300t^2$ edges between vertices in $A'$. Restrict to the graph induced on $A'$, and note that because $|A'| \leq 42t$ and $t \geq 49$, this graph has average degree at least $(2 \cdot 300t^2)/(42t)=14t+2t/7\geq 14t + 14$. We then construct a subgraph of minimum degree at least $7t + 7$, which contains a $K_t$-immersion by \cref{thm:MinDegImmersion}. Note that in both cases, we have obtained a totally even $K_t$-immersion of $G$ with branch vertices contained in $A'$ and with branch trails edge-disjoint from the odd circuits of $G$ corresponding to the $\lceil t/2 \rceil$ loops on each vertex in $A'$.

    We now obtain a totally odd $K_t$-immersion as follows. First we orient the edges of $K_t$ such that each vertex has in degree and out degree at most $\lceil t/2 \rceil$. This orientation can be obtained by first removing a perfect matching if $t$ is even, and then orienting each cycle in a cycle decomposition of the remaining Eulerian graph. We then let each vertex use one loop for each outgoing edge of the orientation to make it an odd trail, resulting in a totally odd $K_t$-immersion. This completes the proof of \cref{thm:makingImmersionOdd}.
\end{proof}

We note that the above construction can be used to find a $K_t$-immersion in which the branch trails have any prescribed parity. Thus we obtain the same bounds as in \cref{thm:MainColoringResult} and \cref{thm:UnionOfBipAndNoImmersion} for graphs which forbid any $K_t$-immersion with a fixed parity of its branch trails.

\section{Forbidding a totally odd $K_t$-immersion}

In this section we prove \cref{thm:UnionOfBipAndNoImmersion} by applying \cref{thm:makingImmersionOdd} and \cref{cor:OddCircuitsEP}. We show that in a maximum bipartite set of edges, the remaining graph cannot have a large clique immersion else we can separate off a bipartite piece of the graph and grow the number of bipartite edges.

\begin{proposition}\label{prop:ImmersionGivesEdgeCut}
    Let $t \geq 49$ and let $G$ be a 2-edge-connected graph with no totally odd $K_t$-immersion. Suppose that $S \subseteq V(G)$ is the branch vertices of a totally even $K_{1400t}$-immersion. Then there exists a set $X \subseteq E(G)$ of size at most $14000t^2$ such that there exists a bipartite component $C$ of $G \setminus X$ containing over half of the vertices in $S$.
\end{proposition}
\begin{proof}
    We first construct a graph $G'$ by adding a vertex $x$ adjacent to every vertex in $S$ with multiplicity $20t$. By \cref{thm:makingImmersionOdd} there are not $3500t^2$ odd $x$-circuits in $G'$, so by \cref{cor:OddCircuitsEP}, there exists a set $X \subseteq E(G')$ of size at most $7000t^2$ such that $G' \setminus X$ has no odd $x$-circuit, which implies that the edge block of $G' \setminus X$ containing $x$ is bipartite. We first note that there are at least $1000t$ vertices in $S$ which are 2-edge-connected to $x$ in $G' \setminus X$, since otherwise $X$ has size at least $(20t-1)(400t) > 7000t^2$. Let $S' \subseteq S$ be a set of $1000t$ vertices which are each 2-edge-connected to $x$, and note that $S'$ is the vertex set of a $K_{1000t}$-immersion in $G$. Also note that every edge block in $G$ that contains a vertex in $S'$ is bipartite.
    
    We now claim that there exists an edge block in $G \setminus X$ that contains at least $900t$ of the vertices in $S'$. We suppose towards a contradiction that there is no edge block containing at least $900t$ of the vertices in $S'$ and count the branch trails of the $K_{1000t}$ immersion which are disjoint from $X$. Note that at most $|S'|-1$ of these branch trails contain a bridge of $G\setminus X$.Thus the maximum number of branch trails which are disjoint from $X$ is
    $$\sum_{B\text{ edge block}} \binom{|B \cap S'|}{2} + |S'| - 1.$$
    Because the size of $S'$ is fixed at $1000t$, the above expression is maximized when $|B \cap S'|$ is as large as possible. If there is no edge block containing $900t$ vertices in $S'$, then the maximum number of branch trails disjoint from $X$ is at most
    $$\binom{900t}{2} + \binom{100t}{2} + 1000t \leq 420000t^2.$$
    This implies that at least $\binom{1000t}{2} - 420000t^2$ edges were deleted, which is much larger than the allowed $7000t^2$.

    Therefore there exists a set of edges $X \subseteq E(G)$ of size at most $7000t^2$ such that there is a bipartite edge block in $G \setminus X$ containing at least $900t$ vertices from $S$. We let $C$ be the vertex set of this edge block. We now claim that $|(E(C) \cap X) \cup \delta(C)| \leq 2|X|$, and hence the desired result follows. Consider the components of the graph obtained from $G\setminus X$ by deleting the vertices in $C$. Exactly $|\delta(C)\setminus X|$-many of these components have an edge to $C$ in $G\setminus X$. Because $G$ is 2-edge-connected, there must be at least one edge of $X \setminus E(C)$ out of each such component. If an edge in $X$ is counted for two such components, then that edge is not in $\delta(C)$, so the bound follows.
\end{proof}

We now use the above proposition to grow a bipartite set of edges in a graph forbidding a totally odd $K_t$-immersion.

\begin{proposition}\label{prop:GrowingBipartite}
    Let $G$ be a graph which does not admit a totally odd $K_t$-immersion. Let $F \subseteq E(G)$ such that $G[F]$ is bipartite. Then either
    \begin{enumerate}
        \item there exists a set of edges $F'$ such that $G[F']$ is bipartite and $|F'| > |F|$, or
        \item $G \setminus F$ does not contain a $K_{98{\small,}000t + 4{\small,}410{\small,}071}$-immersion.
    \end{enumerate}
\end{proposition}
\begin{proof}
    Clearly we may assume $t \geq 4$ else $G$ is bipartite. We may assume $G$ is connected, and in fact we may assume $G$ is 2-edge-connected because an immersion of a clique with at least three vertices must be entirely contained in a single edge block. Let $t' = t + 45$, so $t' \geq 49$ and $G$ forbids a totally odd $K_{t'}$-immersion. Suppose that $G \setminus F$ contains a $K_{98{\small,}000t + 4{\small,}410{\small,}071}$-immersion. Then by \cref{lem:MakingTotallyEven}, $G \setminus F$ contains a totally even $K_{1400t'}$-immersion. By \cref{prop:ImmersionGivesEdgeCut}, there exists a set $X \subseteq E(G)$ of size at most $14000(t')^2$ such that there exists a bipartite component $C$ of $G \setminus X$ containing over half of the branch vertices of the $K_{1400t'}$-immersion. We set $F' = (F \cup E(C)) \setminus X$. Note that $G[F']$ is bipartite. We claim that $|F'| > |F|$. Note that $E(C) \setminus F$ must contain at least $\binom{700t'}{2} - |X|$ edges from the branch trails of the $K_{1400t'}$-immersion. Thus
    $$|F'| - |F| \geq \binom{700t'}{2} - 2|X| \geq \binom{700t'}{2} - 28000(t')^2 > 0$$
    as desired.
\end{proof}

\cref{prop:GrowingBipartite} immediately implies \cref{thm:UnionOfBipAndNoImmersion} by setting $F$ to be a maximum size set of edges such that $G[F]$ is bipartite. This gives \cref{thm:MainColoringResult} as an immediate corollary. We restate \cref{thm:MainColoringResult} below for convenience.
\MainColoringTheorem*
\begin{proof}
    Let $G$ be a graph not admitting a totally odd $K_t$-immersion. By \cref{thm:UnionOfBipAndNoImmersion}, we can write $G = G_1 \cup G_2$ where $G_1$ is bipartite and $G_2$ forbids a $K_{98{\small,}000t + 4{\small,}410{\small,}071}$-immersion. Then we can color $G_1$ with two colors, and by \cref{thm:coloringNoImmersion}, we can color $G_2$ with $346{\small,}920t + 15{\small,}611{\small,}655$ colors. Taking the product of these two colorings we obtain a proper coloring of $G$ with at most $700{\small,}000t + 32{\small,}000{\small,}000$ colors.
\end{proof}

\section{Algorithms}

Our proof of \cref{thm:UnionOfBipAndNoImmersion} can be made into an algorithm for finding such a decomposition, giving \cref{thm:DecompositionAlgorithm}. A crucial piece of our algorithm is a fixed-parameter tractable algorithm for immersion testing due to Grohe, Kawarabayashi, Marx, and Wollan~\cite{Grohe2011Finding}.

\begin{theorem}[{\cite[Corollary~1.2]{Grohe2011Finding}}]\label{thm:ImmersionTesting}
    There exists an algorithm that takes as input simple graphs $G$ and $H$ and decides whether $G$ contains $H$ as an immersion in $\mathcal{O}_{H}(|V(G)|^3)$ time.
\end{theorem}

This can be made into an algorithm for finding an immersion. We will only need the case where $H$ is a clique.

\begin{corollary}\label{cor:FindingImmersionAlg}
    There exists an algorithm that takes as input a mulitgraph $G$ and an integer $t \geq 1$ and returns the branch vertices and branch trails of a $K_t$-immersion in $G$ if one exists in $\mathcal{O}_t(|V(G)|^4 + |E(G)|)$ time.
\end{corollary}
\begin{proof}
    We assume that $t \geq 3$ else we can return some vertex or edge of $G$. We first note that if $G$ is a multigraph with $|E(G)| = \mathcal{O}_t(|V(G)|)$, then we can test for a $K_t$-immersion in $\mathcal{O}_t(|V(G)|^3)$ time. To see this, consider subdividing every edge to create a simple graph. Because $t \geq 3$, a $K_t$-immersion in this simple graph can easily be translated to an immersion in $G$. Hence we can apply \cref{thm:ImmersionTesting}.

    The algorithm is as follows. First, we delete edges from $G$ until there are no loops and every edge has multiplicity at most $\binom{t}{2}$. We then delete edges until $|E(G)|$ is at most $\binom{t}{2}(7t+7)|V(G)|$. We test whether $G$ contains a $K_t$-immersion, and if not then we are done.

    We now find the $K_t$-immersion in $G$. We will iteratively split off and delete edges. At each step, an edge stores a trail in the original graph. When we split off two edges, we concatenate their corresponding trails.

    We iterate through $V(G)$ and find all vertices of degree at most $\binom{t}{2}(8t+7)$. Call the set of such vertices $S \subseteq V(G)$, and initialize a set of branch vertices $B = \varnothing$. We maintain that every $K_t$-immersion in $G$ has $B$ as a subset of its branch vertices, and every vertex in $B$ has degree $t-1$.

    We now repeat the following while $V(G) \neq B$. If $S \setminus B$ is empty, then we delete an arbitrary edge of $G$. We add one or both of the ends of this edge to $S$ if its degree becomes at most $\binom{t}{2}(8t+7)$. Else if there exists a vertex $ v \in S \setminus B$, we delete every edge incident to $v$ while doing so does not destroy the $K_t$-immersion, possibly adding the other end to $S$. Then, for each pair of edges incident to $v$, we split them off when doing so does not destroy the $K_t$-immersion. If $v$ becomes isolated then we delete $v$, otherwise we add $v$ to $B$.

    If $V(G) = B$, then $G$ is isomorphic to $K_t$, and we read off the branch vertices and branch trails from the vertices and edges of $G$.

    We now prove the correctness of the above algorithm and analyze its runtime. If $G$ contains a $K_t$-immersion, then we may assume that all branch trails are paths. Thus deleting loops and reducing the multiplicity of each edge to $\binom{t}{2}$ does not change whether $G$ contains a $K_t$-immersion. We then claim that if $G$ contains at least $\binom{t}{2}(7t+7)|V(G)|$ edges, then $G$ has a $K_t$-immersion. To see this, consider deleting all but one edge from each parallel class; the resulting subgraph still has $(7t+7)|V(G)|$ many edges. We can then find a subgraph of minimum degree $7t+7$ and apply \cref{thm:MinDegImmersion}. Thus the first preprocessing step does not change whether $G$ has a $K_t$-immersion, and afterwards we have $|E(G)| = \mathcal{O}_t(|V(G)|)$. This step can be done in $\mathcal{O}_t(|E(G)|)$ time.

    We now consider the main loop of the algorithm. First note that each time we delete an edge, we possibly add its endpoints to $S$. Thus $S$ need only be computed in $\mathcal{O}_t(|V(G)|)$ time once. We maintain that every edge has multiplicity at most $\binom{t}{2}$ throughout the algorithm. To see this, note that if splitting off a pair of edges created an edge of multiplicity larger than $\binom{t}{2}$, then those edges would have been deleted instead.
    
    Consider when $S \setminus B$ is empty. Because each edge has multiplicity at most $\binom{t}{2}$ and $|B| \leq t$, $G\setminus B$ has minimum degree greater than $\binom{t}{2}(7t+7)$. This implies $B = \varnothing$, otherwise $G$ has a $K_t$-immersion not using a vertex in $B$ by deleting all but one edge in each parallel class and applying \cref{thm:MinDegImmersion}, contradicting that the vertices in $B$ are the branch vertices of every $K_t$-immersion in $G$. Thus after deleting an arbitrary edge, $G$ still has minimum degree at least $\binom{t}{2}(7t+7)$ implying that $G$ contains a $K_t$-immersion.

    Suppose $S \setminus B$ is not empty, and let $v$ be the chosen vertex in $S \setminus B$. We can test every way of deleting an edge incident to $v$ and splitting off a pair of edges at $v$ in $\mathcal{O}_t(|V(G)|^3)$ time because $v$ has degree $\mathcal{O}_t(1)$. If $v$ does not become isolated, then every way of deleting an edge incident to $v$ and splitting off a pair of edges at $v$ must destroy every $K_t$-immersion in $G$. Hence $v$ is contained in the branch set of every $K_t$-immersion in $G$ and has degree exactly $t-1$, implying that $B$ maintains the desired properties.

    Because every vertex eventually appears in $S$, the total runtime of this loop is $\mathcal{O}_t(|V(G)|^4)$. Together with the preprocessing phase of deleting edges, we get the desired runtime.
\end{proof}

We now prove \cref{thm:DecompositionAlgorithm}, which we restate below for convenience.

\DecompositionAlgorithmThm*
\begin{proof}
     We first note that one of the options in \cref{lem:OddPathsEP} can be found in $\mathcal{O}_k(|V(G)||E(G)|)$ time for simple graphs $G$ as discussed in \cite[Section~4]{Geelen2009oddMinor} and using the $\mathcal{O}(\sqrt{|V(G)|}|E(G)|)$ maximum matching algorithm of Micali and Vazirani~\cite{MicaliV1980algorithmfor}. This implies that we can find either outcome of \cref{cor:OddCircuitsEP} in $\mathcal{O}_k(|E(G)|^3)$ time via applying \cref{lem:OddPathsEP} to the $1$-subdivision of the line graph as in the proof of \cref{cor:OddCircuitsEP}.
     
    We can find the set $X \subseteq E(G)$ and component $C$ of $G \setminus X$ as in \cref{prop:ImmersionGivesEdgeCut} in $\mathcal{O}_t(|E(G)|^3)$ time by applying \cref{thm:makingImmersionOdd} and the algorithmic version of \cref{cor:OddCircuitsEP}. This means we can find the set of edges $F'$ in \cref{prop:GrowingBipartite} if it exists in $\mathcal{O}_t(|V(G)|^4 + |E(G)|^3)$ time as follows. We first decompose $G$ into its edge blocks and add all bridges of $G$ to $F'$. We do the following for each edge block $G'$. We find a $K_{98{\small,}000t + 4{\small,}410{\small,}071}$-immersion in $G' \setminus F$ if it exists via \cref{cor:FindingImmersionAlg} in $\mathcal{O}_t(|V(G')|^4 + |E(G')|)$ time. We then apply \cref{lem:MakingTotallyEven} and the algorithmic version of \cref{prop:ImmersionGivesEdgeCut} to find our edge set $F'$ in $\mathcal{O}_t(|E(G')|^3)$ time.
    
    By applying \cref{prop:GrowingBipartite} at most $|E(G)|$ times, we obtain the desired decomposition in $\mathcal{O}_t(|V(G)|^4|E(G)| + |E(G)|^4)$ time.
\end{proof}

\noindent
\textbf{Acknowledgments.}
The author would like to thank Rose McCarty and Paul Wollan for their helpful discussions.

\bibliographystyle{abbrv}
\bibliography{refs}
\end{document}